\documentclass[acmtocl]{acmtrans2m}

\acmVolume{2}
\acmNumber{3}
\acmYear{01}
\acmMonth{09}

\newcommand{\BibTeX}{{\rm B\kern-.05em{\sc i\kern-.025em b}\kern-.08em
    T\kern-.1667em\lower.7ex\hbox{E}\kern-.125emX}}

\usepackage{graphics,color}
\usepackage{graphicx}
\usepackage{latexsym}
\usepackage{amsmath}
\usepackage{amsfonts}
\usepackage{amssymb}
\usepackage{url}
\usepackage{subfigure}
\usepackage{times}

\numberwithin{equation}{section}

\newtheorem{proposition}{Proposition}
\newtheorem{theorem}{Theorem}

\newtheorem{corollary}{Corollary}

\newcommand{\E}{{\sf E}}

\newcommand{\diag}{\text{\text{diag}}}

\newcommand{\secure}{{\sf susceptible}}
\newcommand{\compromised}{{\sf infected}}

\newcommand{\susceptible}{{\sf susceptible}}
\newcommand{\infected}{{\sf infected}}

\newcommand{\cc}{i}

\newcommand{\oneton}{1,\cdots,n}

\title{Adaptive Epidemic Dynamics in Networks: Thresholds and Control}

\author{Shouhuai Xu\\University of Texas at San Antonio\\
Wenlian Lu\\Fudan University \and
Li Xu and Zhenxin Zhan\\University of Texas at San Antonio}

\begin{abstract}
Theoretical modeling of computer virus/worm epidemic dynamics
is an important problem that has attracted many studies. However, most existing models are
adapted from biological epidemic ones. Although biological epidemic models
can certainly be adapted to capture some computer virus spreading scenarios (especially when the
so-called homogeneity assumption holds), the problem of
computer virus spreading is not well understood because it has many important perspectives
that are not necessarily accommodated in the biological epidemic models.
In this paper we initiate the study of such a perspective, namely
that of {\em adaptive} defense against epidemic spreading in arbitrary
networks. More specifically, we investigate a non-homogeneous
Susceptible-Infectious-Susceptible (SIS) model where
the model parameters may vary with respect to time.
In particular, we focus on two scenarios we call {\em semi-adaptive} defense
and {\em fully-adaptive} defense, which accommodate implicit and explicit
dependency relationships between the model parameters, respectively.
In the semi-adaptive defense scenario, the model's input parameters are given;
the defense is semi-adaptive because the adjustment is implicitly dependent upon the outcome of virus spreading.
For this scenario, we present a set of sufficient conditions (some are more general or succinct than others)
under which the virus spreading will die out;
such sufficient conditions are also known as {\em epidemic thresholds} in the literature.
In the fully-adaptive defense scenario, some input parameters are not known (i.e.,
the aforementioned sufficient conditions are not applicable)
but the defender can observe the outcome of virus spreading.
For this scenario, we present adaptive control strategies under which
the virus spreading will die out or will be contained to a desired level.
\end{abstract}

\category{K.6.5}{Management of Computing and Information Systems}
{Security and Protection}

\terms{Security}

\keywords{computer malware, virus epidemics, epidemic dynamics, epidemic threshold, complex network, graph}

\begin{document}

\begin{bottomstuff}
Author's address: Shouhuai Xu, Li Xu, and Zhenxin Zhan are with the Department of Computer Science,
University of Texas at San Antonio. Corresponding author: Shouhuai Xu ({\tt shxu@cs.utsa.edu}).
\newline
Wenlian Lu is with the Center for Computational Systems Biology and the
School of Mathematical Sciences, Fudan University ({\tt wenlian@fudan.edu.cn}).
\end{bottomstuff}

\maketitle

\section{Introduction}

Theoretically modeling the spreading dynamics of computer virus (or malware such as worm and bot)
is important for deepening our understanding and for
designing effective, if not optimal, defenses.
We observe, however, that the utility of theoretical modeling in this context is not well understood yet
because existing models are often adapted from biological epidemic ones.
As a consequence, many existing models of computer virus spreading dynamics made the so-called {\em homogeneity}
assumption, which roughly says that the nodes are equally powerful in infecting others.
Realizing the limitation of the assumption, there have been investigations
that aim to weaken the assumption by considering heterogeneous network
topology (where different nodes may have different infection capabilities
because they have different degrees).
Along this line of study, the present paper moves a step further by
exploring models that accommodate realistic scenarios where
the model parameters may change over time (i.e., the parameters are some
functions of time), which captures the fact that both attack and defense are dynamically evolving
or under dynamical adjustment and reflects the persistence of virus spreading.
This allows us to investigate an
important and novel perspective of virus spreading-defense dynamics,
namely that of {\em adaptive} defense against computer virus spreading.

\subsection{Our Contributions}

We investigate a non-homogeneous Susceptible-Infectious-Susceptible (SIS) model in
arbitrary networks (i.e., there is no restriction on the topology of the spreading networks
and the nodes may have different defense or cure capabilities).
The model can accommodate both {\em semi-adaptive} defense and {\em fully-adaptive} defense.
In the semi-adaptive defense scenario, the input parameters in the model are known and can vary
with respect to time (e.g., according to some
deterministic functions of time or according to some stochastic process, but we do not impose
any practical restrictions on the types of functions).
For this scenario, we present a set of sufficient conditions, from general to specific (but more succinct),
under which the virus spreading will die out. We note that
such sufficient conditions are also known as {\em epidemic thresholds} in the literature.

In the fully-adaptive defense scenario, some input parameters are not known and thus
the aforementioned sufficient conditions are not applicable. Nevertheless,
the defender might be able to observe the outcome of virus spreading (i.e., which nodes are
infected at a point in time).
For this scenario, we present adaptive control strategies
under which the virus spreading will die out or will be contained to a desired level
(which is important when, for example, the price to kill the virus spreading may be too high).

Because of the above, our model supersedes previous homogeneous and non-homogeneous models
that offered relevant analytical insights;
the concrete connection will be made when the need arises.
Our analytical results are confirmed via simulation, from which we draw additional observations
that serve as hints for future modeling studies. We discuss the
practical implications of our model and the derived insights as well.

Finally, we note that the present paper is meant to explore
theoretical characterizations of spreading-defense dynamics
while assuming certain parameters can be observed or measured
(e.g., based on extensive data and possibly expert knowledge).
This may not be feasible some times.
Regardless, we believe that such studies are important on their own and represent a necessary
step towards the ultimate characterization of virus spreading-defense dynamics
(which in turn helps design more effective or even optimal defenses).

\subsection{Related Work}

To the best of our knowledge, there are no existing studies on modeling
adaptive spreading-defense dynamics in arbitrary networks.
The work that is most closely related to ours was due to Chakrabarti et al. \cite{WangTISSEC08},
who considered computer virus spreading in arbitrary networks --- a scenario also investigated
in \cite{WangSRDS03,TowsleyInfocom05}. The most important contribution of these studies is
the identification of a sufficient condition (i.e., epidemic threshold) under which the virus spreading will die out;
we will discuss the relationship between their result and ours when the need arises.
Earlier studies either made the homogeneity assumption
(e.g., \cite{KephartOkland91,KephartOkland93}) as in
biological epidemic models (see, for example, \cite{McKendrick1926,Kermack1927,Bailey1975,Anderson1991,HethcoteSIAMRew00}),
or considered specific non-homogeneous networks \cite{WangTISSEC08}.

We should mention prior work that is conceptually or spiritually relevant.
The concept of ``adaptable robust computer systems" was investigated by Bhargava et al. \cite{BhargavaSIGOPS86},
which however has a very different meaning and is for very different purposes.
Also for a different purpose, Zou et al. \cite{ZouSRUTI05} explored the concept of ``adaptive defense" based on cost optimization,
where cost was introduced by false positives and false negatives. In particular, they considered
optimal adaptive defense against worm infection, but is from the perspective of
decision whether or not to block/allow some specific host traffic. As such, it may be possible to
combine their studies and ours because we do not consider cost.

\smallskip

\noindent{\bf Outline}: In Section \ref{sec:model-alpha-eq-0} we present our model as well as the
analytical insights. We report our simulation study in Section \ref{sec:simulation-study-alpha-eq-0}.
We conclude the paper in Section \ref{sec:conclusion} with open problems.

\section{Adaptive Epidemic Dynamics: Model and Analysis}
\label{sec:model-alpha-eq-0}

\subsection{The model}

\noindent{\bf Primary parameters}.
Because we want to accommodate spreading in arbitrary networks, we assume that virus spreads over
a series of finite, dynamical graphs $G(t)=(V,E(t))$, where
$V$, $|V|=n$, is the set of
nodes or vertices and $E(t)$ is the set of (possibly changing) edges or arcs at time $t \ge 0$
(i.e., the topology may change with respect to time).
At any time $t$, an infected node $u$ can directly infect node $v$ if $(u,v)\in E(t)$.
Denote by $A(t)=[a_{vu}(t)]$ the adjacency matrix of $G(t)$, where
$a_{vv}(t)=0$ for all $v\in V$, and $a_{vu}(t)=1$ if and only if $(u,v) \in E(t)$.
Note that this representation naturally accommodates both directed and undirected topologies,
and thus our results equally apply to them.

A node $v \in V$ is \susceptible\ if $v$ is secure but
vulnerable, and \infected\ if $v$ is successfully attacked (i.e., infected and infectious).
At any time $t$, a node $v \in V$ is either \susceptible\ or \infected.
Moreover, a \susceptible\ node may become
\infected\ because of some \infected\ node $u$ where
$(u,v)\in E(t)$, and an \infected\ node may become \susceptible\ because of cure. Since an
\infected\ node may become \susceptible\ again, our model falls into the category of the
so-called SIS models, but our model has the unique feature that
values of the parameters can change with respect to time.

We consider two dependent variables:
$s_v{(t)}$, the probability $v \in V$ is \susceptible\ at time
$t$; $\cc_{v}{(t)}$, the probability $v\in V$ is \infected\ at
time $t$. We consider a continuous-time model, which preserves the invariant
$s_v{(t)} + \cc_{v}{(t)}=1$. The model's input parameters are:
\begin{itemize}
\item $\beta_v(t)$: The probability an \infected\ node $v$ becomes \susceptible\ at time $t$.

\item $\gamma_{uv}(t)$: The probability an \infected\ node $u$ successfully infects a \susceptible\
node $v$ over edge $(u,v)\in E(t)$ at time $t$.
For simplicity, we assume that $\gamma_{uv}(t) = \gamma(t)$ for all $(u,v) \in E(t)$.
\end{itemize}
For the sake of mathematical rigorousness, the $\beta_v(t)$'s and the $\gamma_{uv}(t)$'s, which are probabilities, should be ``measurable"
so as to ensure the existence of solutions to system (\ref{Eq.2.2}) and be
``bounded" so as to ensure the proof of Theorem \ref{thm4} can get through.
To avoid any unnecessary mathematical subtleties, we simply assume that
these parameters are ``boundedly measurable," which has no consequence in practice.
Note that $A(t)$ is naturally bounded.

\smallskip

\noindent{\bf Other parameters and notations}.
Below is a summary of the major notations used in the paper;
notations only occasionally used are explained when the need arises.

\begin{center}
\begin{tabular}{|r|p{.6\textwidth}|}
\hline
model input parameters: &  \\
$A(t)=[a_{vu}(t)]$      & the adjacency matrix of graph $G(t)=(V,E(t))$ where
                          $|V|=n$, and $a_{vu}(t)=1$ if and only if $(u,v)\in E(t)$. Moreover, $a_{vv}(t)=0$ for all $v \in V$.\\
$\beta_{v}(t)$          & the cure capability of node $v$ at time $t$\\
$\gamma(t)$             & the edge infection capability at time $t$\\
\hline
dependent variables: &\\
$s_v{(t)}$          & the probability node $v\in V$ is \secure\ at time $t$ \\
$\cc_{v}{(t)}$      & the probability node $v\in V$ is \compromised\ at time $t$ \\
\hline
intermediate variables: &\\
$\delta_{v}{(t)}$   & the probability \susceptible\ $v\in V$ becomes \infected\ at time $t$ because of \compromised\
                      neighbors $\{u: (u,v) \in E(t)\}$ \\
\hline
other notations: &\\
$\lambda_1$             & the largest (in modulus) eigenvalue of adjacency matrix $A$\\
$\|\cdot\|$ & the 1-norm of vector or matrix \\
$B(t)=\diag[\beta_{1}(t),\cdots,\beta_{n}(t)]$        & the cure probability diagonal matrix \\
\hline
\end{tabular}
\end{center}

\noindent{\bf The state transition diagram and master equation}.
Figure \ref{fig:state-diagram-research-plan-adaptive-alpha=0} depicts the state transition diagram of a node, where
the probability $\delta_v{(t)}$ is given by
\begin{equation}
\label{compromise}
\delta_v {(t)}=1-\prod\limits_{(u,v)\in E(t)}\left[1-\gamma(t) \cdot i_u {(t)}\right].
\end{equation}

\begin{figure}[ht]
\centering
\includegraphics[width=.8\textwidth]{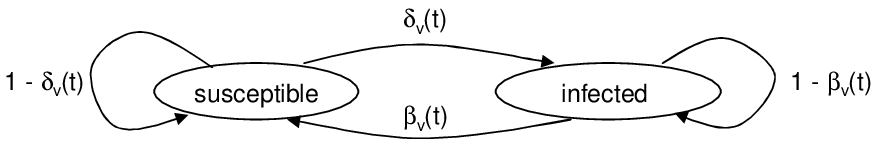}
\caption{State transition diagram of node $v \in V$ at time $t$}
\label{fig:state-diagram-research-plan-adaptive-alpha=0}
\end{figure}

Note that in the derivation of Eq. (\ref{compromise}), we assumed that the events that \infected\ neighbors infect
a node are independent.
Note also that $s_v(t) + \cc_v(t)=1$ for any $t$.
Based on the state transition diagram we obtain the following master equation of dynamics:
\begin{eqnarray}\label{Eq.2}
\frac{d\cc_{v}(t)}{dt} =\bigg[1-\prod_{u\in V}[1-\gamma(t)
a_{vu}(t) \cc_{u} {(t)}]\bigg][1-\cc_{v} {(t)}]-\beta_{v}(t)\cc_{v}
{(t)},
\end{eqnarray}
where $a_{vu}(t)=1$ if and only if $(u,v)\in E(t)$.

\subsection{Sufficient conditions for dying out in the scenario of semi-adaptive defense}
\label{sec:implicit-adaptive-case}

In this subsection we present a set of sufficient conditions under which the virus spreading will die out.
The sufficient conditions are applicable when the model's input parameters,
namely the $\beta_v(t)$'s and $\gamma(t)$ are given. Moreover, it is possible
that $\beta_v(t)$ relies on $\gamma(t)$; for example, the former is an (implicit) function of the latter.
This explains why we call this scenario the semi-adaptive defense.

\begin{theorem}
\label{theorem-1}
\emph{(a general sufficient condition under which virus spreading dies out)}
Consider the following comparison (and linearization) system of Eq. (\ref{Eq.2}):
\begin{eqnarray}
\label{Eq.2.1}
\frac{d x_{v}(t)}{dt} =\sum_{u=1}^{n}a_{vu}(t)\gamma(t)x_{u}(t) -\beta_{v}(t)x_{v} {(t)}.
\end{eqnarray}
Let
$x(t)=[x_{1}(t),\cdots,x_{n}(t)]^{\top}$, we obtain the following compact form of Eq. (\ref{Eq.2.1}):
\begin{eqnarray}\label{Eq.2.2}
\frac{d x(t)}{dt} =\bigg[\gamma(t)A(t) -B(t)\bigg]x(t).
\end{eqnarray}
Denote by $U(t,t')$ the solution matrix of linear system
(\ref{Eq.2.2}), namely that each solution of
linear system (\ref{Eq.2.2}), $x(t)$, with initial condition
$x(t_{0})=x_{0}$, can be written as $x(t)=U(t,t_{0})x_{0}$. Because
the solution $x(t)$ is dependent upon the initial value $x_0$ but
the solution matrix $U(t,t')$ is not, the corresponding
maximum Lyapunov exponent (MLE) is determined by
the solution matrix (rather than by the solution $x(t)$) and can be defined as \cite{Ose}:
\begin{eqnarray*}
\mu=\lim_{t\to\infty}\frac{1}{t}\ln\|U(t,0)\|
\end{eqnarray*}
where $\|\cdot\|$ is (for specificality) the 1-norm of matrix (because $\mu$ is
independent of the choice of the norm).
If $\mu<0$, the virus spreading will die out regardless of the initial infection
configuration;
if $\mu>0$ and system (\ref{Eq.2.2}) is ergodic \cite{Ose}, the virus spreading will not die out in some initial infection
configurations (i.e., ``the equilibrium of $i^*=0$ is unstable" in mathematical terms).
\end{theorem}

\begin{proof}
Note that
\begin{eqnarray}
1-\prod_{u\in V}[1-\gamma(t)a_{vu}(t)\cc_{u}(t)]\le
\gamma(t)\sum_{u\in V}a_{vu}(t)\cc_{u}(t).\label{ineq}
\end{eqnarray}
If $\cc_{v}(0)=x_{v}(0)$ for all $v\in V$, then the
comparison system (\ref{Eq.2.1}) satisfies that
$\cc_{v}(t)\le x_{v}(t)$ holds for all $t\ge 0$ and $v\in V$.
Since Eq. (\ref{Eq.2.1}) is actually the
linear system of the system (\ref{Eq.2}) at $\cc_{v}=0$ for all
$v\in V$. Therefore, we can conclude that the stability
of Eq. (\ref{Eq.2}) is equivalent to that of Eq. (\ref{Eq.2.1}). In
other words, if all $x_{v}(t)$'s of system (\ref{Eq.2.1}) converge
to zero, then system (\ref{Eq.2}) is stable regardless of the
initial values; on the other hand, if system (\ref{Eq.2.1}) is
unstable, then system (\ref{Eq.2}) is also unstable. If system
(\ref{Eq.2.2}) is ergodic \cite{Ose}, then the limit $\mu$ exists;
otherwise, we can alternatively define
$\mu=\overline{\lim}_{t\to\infty}\frac{1}{t}\ln\|U(t,0)\|$,
where $\overline\lim_{t\to\infty} z(t)$ represents the upper bound
of the limit of $z(t)$ as $t$ goes to infinity, is also guaranteed
to exist. In any case, by applying the definition of MLE, we obtain
the theorem immediately.
\end{proof}

\noindent{\bf Discussion}. The above sufficient condition $\mu<0$
for the virus spreading to die out is actually close to being
necessary, meaning that if $\mu >0$, then the virus spreading will
not die out in {\em most}, rather than just {\em some}, initial
infection configurations. According to the Lyapunov exponent and smooth
ergodic theory developed by \cite{Pesin77} and many
others, $\mu>0$ means that the system (\ref{Eq.2}) possesses an
unstable manifold, which implies that the stable manifold, i.e., the
set of points (i.e., the initial values) starting from which the system (\ref{Eq.2})
converges to the origin, has dimension less than $n$. Therefore,
the stable manifold has Lesbegue measure $0$. That is, except a set
with Lesbegue measure $0$, the virus spreading
never dies out with respect to any initial infection configuration.

The sufficient condition given in Theorem \ref{theorem-1} is very
general because in its derivation no ``amplification" is used and
dynamical topology $E(t)$ is accommodated. However, it requires to,
among other things, solve a system of $n$ linear equations of $n$
variables (equivalently, diagonalizing a $n \times n$ matrix), which
can be quite time-consuming for large $n$ (the number of nodes).
In what follows we give two succinct sufficient conditions,
which can be easily connected to previous state-of-the-art results.

\begin{theorem}
\label{thm4}
\emph{(a succinct sufficient condition)} Suppose
$\beta_v(t)=\beta(t)$ for all $v \in V$ (i.e., all nodes have the
same cure capability) and $E(t)=E$ for any time $t$, meaning that the topology
does not change over time and $A=A(t)$ for any $t$. Consider the system
\begin{eqnarray}\label{Eq.2.3}
\frac{d x(t)}{dt} =\bigg[\gamma(t)A -\beta_v(t)I_{n}\bigg]x(t),
\end{eqnarray}
where $I_n$ is the $n \times n$ identity matrix. Let
\begin{eqnarray*}
\bar{\gamma}=\lim_{t\to\infty}\frac{1}{t}\int_{t_{0}}^{t+t_{0}}\gamma(\tau)d\tau,
~~~\bar{\beta}=\lim_{t\to\infty}\frac{1}{t}\int_{t_{0}}^{t+t_{0}}\beta(\tau)d\tau,
\end{eqnarray*}
and suppose the limits exist and are uniform with respect to
$t_{0}$. Let
$\lambda_{1}$ be the largest (in modulus) eigenvalue of $A$.
 If
\begin{eqnarray}
\lambda_{1}<\frac{\bar{\beta}}{\bar{\gamma}},\label{stable}
\end{eqnarray}
then the virus spreading will die out regardless of the initial infection configuration;
if
\begin{eqnarray}
\lambda_{1}>\frac{\bar{\beta}}{\bar{\gamma}},\label{unstable}
\end{eqnarray}
then the virus spreading will not die out in some initial infection configurations.
\end{theorem}

\begin{proof}
Consider system (\ref{Eq.2.2}). Let
$A=S^{-1}JS$ be the Jordan canonical form with
\begin{eqnarray*}
J=\left[\begin{array}{lllll}J_{1}&&&&\\
&J_{2}&&&\\&&J_{3}&\\&&&\ddots&\\&&&&J_{K}\end{array}\right],
~J_{k}=\left[\begin{array}{lllll}\lambda_{k}&0&0&\cdots&0\\
1&\lambda_{k}&0&\cdots&0\\0&1&\lambda_{k}&\cdots&0\\
\vdots&\vdots&\vdots&\ddots&\vdots\\0&0&\cdots&1&\lambda_{k}\end{array}\right],
\end{eqnarray*}
where $\lambda_{k}$, $k=1,\cdots,K$, are the distinct eigenvalues of
$A$. Recall that $\lambda_{1}$ is the largest eigenvalue in modulus.
From the Perron-Frobenius theorem \cite{Berman03},  $\lambda_{1}$ is
a real number. Then, letting $y(t)=Sx(t)$, we have
\begin{eqnarray*}
\frac{d y(t)}{dt}=[\gamma(t) J-\beta(t) I_{n}]y(t).
\end{eqnarray*}
Namely,
\begin{eqnarray*}
\frac{d y_{v}(t)}{dt}=[\gamma(t)\lambda_{k_{v}}-\beta(t)]y_{v}(t)+\xi_{v}(t),
\end{eqnarray*}
where $\lambda_{k_{v}}$ is the eigenvalue of $A$ corresponding to
the Jordan block $J_{k_{v}}$ that contains column $v$, and
$\xi_{v}(t)=0$ if the $v$-th row of $S$ is an eigenvector of $A$ and
$\xi_{v}(t)=\gamma(t)y_{v-1}(t)$ otherwise. First, consider $v=1$
corresponding to the eigenvalues $\lambda_{1}$. We have
\begin{eqnarray}
y_{1}(t)=y_{1}(0)\exp\left(\int_{0}^{t}[\gamma(\tau)\lambda_{1}-\beta(\tau)]d\tau\right).\label{comp1-1}
\end{eqnarray}
One can see that the Lyapunov exponent of system (\ref{comp1-1})
is calculated as
\begin{eqnarray*}
\lim_{t\to\infty}\frac{1}{t}\ln||y_{1}(t)||&=&
\lim_{t\to\infty}\frac{1}{t}\ln||y_{1}(0)||+\lim_{t\to\infty}\frac{1}{t}\int_{0}^{t}
[\gamma(\tau)\lambda_{1}-\beta(\tau)]d\tau\\
&=&0+\bar{\gamma}\lambda_{1}-\bar{\beta}<0,
\end{eqnarray*}
which implies $\lim_{t\to\infty}y_{1}(t)=0$.

Assuming $\lim_{t\to\infty}y_{v}(t)=0$ already proved, consider $y_{v+1}(t)$. We have
\begin{eqnarray*}
y_{v+1}(t)=y_{v+1}(0)\exp\left(\int_{0}^{t}k(\tau)d\tau\right)
+\int_{0}^{t}\xi_{v+1}(a)\exp\left(\int_{a}^{t}k(\tau)d\tau\right)
d a
\end{eqnarray*}
where $k(\tau)=\gamma(\tau)\lambda_{k_{v+1}}-\beta(\tau)$.
From condition (\ref{stable}), there exists a sufficiently
small $\epsilon>0$ such that
$(\bar{\gamma}+\epsilon)\lambda_{1}-(\bar{\beta}+\epsilon)<0$.
Let
$\varphi=-(\bar{\gamma}+\epsilon)\lambda_{1}+(\bar{\beta}+\epsilon)$.
One can see that $\varphi>0$. Since the limits
$\int_{t_{0}}^{t_{0}+t}\beta(\tau)d\tau$ and
$\int_{t_{0}}^{t_{0}+t}\gamma(\tau)d\tau$ are uniform with respect
to $t_{0}$, there exists $T>0$ such that
\begin{eqnarray*}
\frac{1}{t-a}\int_{a}^{t}\beta(\tau)d\tau<\bar{\beta}+\epsilon,
~\frac{1}{t-a}\int_{a}^{t}\gamma(\tau)d\tau<\bar{\gamma}+\epsilon
\end{eqnarray*}
hold for any $a$ and $t$ with $t-a>T$. Let ${\mathcal Re} (z)$ denote the real part of a complex number
$z$. Then, we have
\begin{eqnarray*}
\mathcal
Re\bigg(\frac{1}{t-a}\int_{a}^{t}k(\tau)d\tau\bigg)&=&\frac{1}{t-a}\int_{a}^{t}
[\gamma(\tau)\mathcal Re(\lambda_{k_{v+1}})-\beta(\tau)]d\tau\\
&\le&\frac{1}{t-a}\int_{a}^{t} [\gamma(\tau)\lambda_{1}-\beta(\tau)]d\tau \\
&\le&(\bar{\gamma}+\epsilon)\lambda_{1}-(\bar{\beta}+\epsilon)=-\varphi<0
\end{eqnarray*}
for all $t$ and $a$ with $t-a>T$. This implies that the
first term $y_{v+1}(0)\exp(\int_{0}^{t}k(\tau)d\tau)$ converges to
zero. Let $M>0$ be a constant such that
$\sup\limits_{\tau,j}|\gamma(\tau)\lambda_{j}-\beta(\tau)|<M$. For the
second term, we have
\begin{eqnarray*}
&&\bigg|\int_{0}^{t}\xi_{v+1}(a)\exp\left(\int_{a}^{t}k(\tau)d\tau\right)da\bigg| \\
&=&\bigg|\int_{0}^{t-T}\xi_{v+1}(a)\exp\left(\int_{a}^{t}k(\tau)d\tau\right)da
+\int_{t-T}^{t}\xi_{v+1}(a)\exp\left(\int_{a}^{t}k(\tau)d\tau\right)da\bigg|\\
&\le&\int_{0}^{t-T}|\xi_{v+1}(a)|\exp\left(-\varphi
(t-a)\right)da+\int_{t-T}^{t}|\xi_{v+1}(a)|\exp(MT)da.
\end{eqnarray*}
In the case of $\xi_{v+1}(t)=0$, we immediately conclude that
$\lim_{t\to\infty}y_{v+1}(t)=0$. Otherwise, according to the
condition and the L'Hospital principle, we have
\begin{eqnarray*}
\lim_{t\to\infty}\int_{0}^{t-T}|\xi_{v+1}(a)|\exp\big[-\varphi (t-a)\big]da &=&
\lim_{t\to\infty}\frac{\int_{0}^{t-T}|\xi_{v+1}(a)|\exp(\varphi a)da}{\exp(\varphi t)} \\
&=&\lim_{t\to\infty}\frac{|\xi_{v+1}(t-T)|\exp(\varphi(t-T))}{\varphi\exp(\varphi t)}\\
&=&\lim_{t\to\infty}\frac{|y_{v}(t-T)\gamma(t-T)|\exp(-\varphi T))}{\varphi}=0,
\end{eqnarray*}
due to the assumption $\lim_{t\to\infty}y_{v}(t)=0$. We also have
\begin{eqnarray*}
\lim_{t\to\infty}\int_{t-T}^{t}|\xi_{v+1}(a)|\exp(MT)da=0.
\end{eqnarray*}
Therefore, we can conclude $\lim_{t\to\infty}y_{v+1}(t)=0$.

Note that condition (\ref{unstable}) implies that system
(\ref{Eq.2.2}) is unstable. Since system (\ref{Eq.2.2}) is in
fact the linearization system of Eq. (\ref{Eq.2}), we can conclude
that system (\ref{Eq.2}) is unstable under condition
(\ref{unstable}). This completes the proof.
\qed
\end{proof}

If $\beta(t)$ and $\gamma(t)$ are ergodic stochastic processes, from
the multiplicative ergodic theory of the random dynamical
systems \cite{Arn}, we have
the following result as a corollary of Theorem \ref{thm4}.

\begin{corollary}
\label{theorem-main-result-2}
\emph{(another succinct sufficient condition)}
Suppose $\beta_v(t)=\beta(t)$ for all $v \in V$ (i.e., all nodes
have the same cure capability) and $E(t)=E$ for any time $t$ (i.e.,
topology does not change over time). Suppose $\{\beta(t)\}_{t \ge
0}$ and $\{\gamma(t)\}_{t \ge 0}$ are ergodic stochastic processes
(i.e., $\beta(t)$ and $\gamma(t)$ are some random variables). Let
$\E(\beta(0))$ and $\E(\gamma(0))$ be the expectations with respect
to the stationary distributions of the respective ergodic stochastic
process. Suppose the convergences
\begin{eqnarray*}
\E(\beta(0))=\frac{1}{t}\int_{t_{0}}^{t_{0}+t}\beta(\tau)d\tau,~
\E(\gamma(0))=\frac{1}{t}\int_{t_{0}}^{t_{0}+t}\gamma(\tau)d\tau,
\end{eqnarray*}
are both uniform with respect to $t_{0}$ {\em almost surely}.
If
\begin{eqnarray*}
\lambda_{1}<\frac{\E(\beta(0))}{\E(\gamma(0))},
\end{eqnarray*}
the spreading will die out almost surely regardless of the initial infection configuration;
if
\begin{eqnarray*}
\lambda_{1}>\frac{\E(\beta(0))}{\E(\gamma(0))},
\end{eqnarray*}
the spreading will not die out in some initial infection configurations.
\end{corollary}

\noindent{\bf Discussion}. The state-of-the-art sufficient
condition for the dying out of virus spreading in an arbitrary network is
$\lambda_1 < \frac{\beta}{\gamma}$, which was given in \cite{WangTISSEC08}.
In the setting of \cite{WangTISSEC08}, the parameters satisfy that $\beta_v(t)=\beta$ for all $v\in V$ and all $t$,
and $\gamma(t)=\gamma$ for all $t$.
As such, their result is clearly a special case of the above Corollary \ref{theorem-main-result-2}
(note that it is guaranteed that $\E(\gamma(0))\neq 0$),
and thus of the above Theorem \ref{thm4}.

\subsection{Adaptive control in the scenario of fully-adaptive scenario}

In the semi-adaptive defense scenario investigated above, we assumed
that the parameters $\gamma(t)$ and $\beta_v(t)$ are given.
What if they are not given?
In what follows we investigate a representative scenario, where
the defender is not given $\gamma(t)$ but can observe $i_v(t)$.
Specifically, we consider two sufficient conditions of adaptive control:
one under which the virus spreading will die out (Section \ref{sec:adaptive-control-1}),
and another under which the virus spreading
will not die out but will be contained to a desired level of infection (Section \ref{sec:adaptive-control-2}).

\subsubsection{Sufficient condition under which the virus spreading dies out under adaptive control}
\label{sec:adaptive-control-1}

The question we ask is: How should the defender adjust the defense,
namely how $\beta_v(t)$ should depend upon $i_v(t)$, so
that the virus spreading will die out? We assume for concreteness that
$\beta_v(0)=0$ for all $v\in V$; this accounts for the worst-case scenario.

\begin{theorem}
\label{thm_stablize1}
\emph{(characterization of adaptive control strategy under which the virus spreading will die out)}
Suppose without loss of generality $\beta_v(0)=0$ for all $v\in V$.
If
\begin{eqnarray*}
\frac{d\beta_{v}(t)}{dt}=\rho\cc_{v}(t),
\end{eqnarray*}
where $\rho$ is an (almost) arbitrary positive constant, then the virus spreading will die out
regardless of the initial infection configuration.
\end{theorem}

\begin{proof}
Define a candidate Lyapunov function with respect to the infectious
probabilities $\cc=[\cc_{1},\cdots,\cc_{n}]^{\top}$ and the cure
capabilities
$\tilde{\beta}=[\beta_{1},\cdots,\beta_{n}]^{\top}$:
\begin{eqnarray*}
V(\cc,\tilde{\beta})=\sum_{v=1}^{n}i_{v}(t)+\frac{1}{2\rho}\sum_{v=1}^{n}(\beta_{v}(t)-\beta_{0,v})^{2}.
\end{eqnarray*}
Let $\beta_{0,v}$, $v=\oneton$, be positive constants satisfying
\begin{eqnarray*}
\beta_{0,v}>(n-1)\sup_{t}\gamma(t)+1,~\forall~v=\oneton
\end{eqnarray*}
owing to the fact that $1\ge \gamma(t)\ge 0$.
Due to inequality (\ref{ineq}), differentiating $V(\cc,\tilde{\beta})$ gives
\begin{eqnarray*}
\frac{dV(i,\tilde{\beta})}{dt}&=&\sum_{v=1}^{n}\left[1-\prod_{(u,v)\in
E{(t)}}(1-\gamma(t)\cc_{u}(t))\right](1-\cc_{v}(t))-\sum_{v=1}^{n}\beta_{v}(t)\cc_{v}(t) \\
&&+\sum_{v=1}^{n}(\beta_{v}(t)-\beta_{0,v})\cc_{v}(t)\\
&\le&\sum_{v=1}^{n}\gamma(t)\sum_{u=1}^{n}a_{vu}(t)\cc_{u}(t)-\sum_{v=1}^{n}\beta_{v}(t)\cc_{v}(t)
+\sum_{v=1}^{n}(\beta_{v}(t)-\beta_{0,v})\cc_{v}(t)\\
&\le&\sum_{v=1}^{n}\gamma(t)\sum_{u=1}^{n}a_{vu}(t)\cc_{u}(t)-\sum_{v=1}^{n}\beta_{0,v}\cc_{v}(t)\\
&\le&\gamma(t)(n-1)\sum_{u=1}^{n}\cc_{u}(t)-\sum_{v=1}^{n}\beta_{0,v}\cc_{v}(t) \\
&\le& -\sum_{v=1}^{n}\cc_{v}(t).
\end{eqnarray*}
According to the LaSalle principle \cite{LaSalle1960}, the
system converges to the largest invariant set
$\{(\cc,\tilde{\beta}):\sum_{v=1}^{n}\cc_{v}=0\}$, which implies
$\lim_{t\to\infty}\cc_{v}(t)=0$ for all $v=\oneton$. \qed
\end{proof}

The above Theorem \ref{thm_stablize1} has the following implications.

\begin{proposition}
\label{proposition-adaptive-1}
We can bound from above the accumulated number of infected nodes in the long run (a node
is counted multiple times if it is infected at multiple points in time) as follows:
\begin{eqnarray}
\label{est1} \int_{0}^{\infty}\sum_{v=1}^{n}\cc(\tau)d\tau\le
\frac{(n-1)^{2}\gamma_{m}}{\rho}+\frac{n-1}{\rho}\sqrt{\sum_{v=1}^{n}\cc_{v}(0)
+\frac{(n-1)^{3}\gamma_{m}^{2}}{2\rho}}
\end{eqnarray}
where $\gamma_{m}=\sup\limits_{t}\gamma(t)$ and
$\beta_{0,m}=\min\limits_{v}\beta_{0,v}$.
\end{proposition}

\begin{proof}
From the proof of Theorem \ref{thm_stablize1}, we have
\begin{eqnarray*}
\frac{dV(i,B)}{dt}\le
[\gamma_{m}(n-1)-\beta_{0,m}]\sum_{v=1}^{n}\cc_{v}(t),
\end{eqnarray*}
which implies
\begin{eqnarray*}
\int_{0}^{t}\sum_{v=1}^{n}\cc_{v}(\tau)d\tau\le\frac{1}{\beta_{0,m}-\gamma_{m}(n-1)}[V(0)-V(t)]\le
\frac{1}{\beta_{0,m}-\gamma_{m}(n-1)}V(0).
\end{eqnarray*}
If $\beta_{v}(0)=0$ for all $v=\oneton$ and all
$\beta_{v,0}$ are the same and are thus denoted by $\beta_{0}$, we have the
following estimation:
\begin{eqnarray*}
\int_{0}^{\infty}\sum_{v=1}^{n}\cc(\tau)d\tau\le\frac{1}{\beta_{0}-\gamma_{m}(n-1)}\big[\sum_{v=1}^{n}
\cc_{v}(0)+\frac{1}{2\rho}(n-1)\beta_{0}^{2}\big].
\end{eqnarray*}
By setting
$$\beta_{0}=a+\frac{\sqrt{\sum_{v=1}^{n}\cc_{v}(0)+1/(2\rho)(n-1)^{3}\gamma_{m}^{2}}}{1/(2\rho)(n-1)},$$
we obtain the minimum of the right-hand side and thus complete the proof. \qed
\end{proof}

\noindent{\bf Physical meanings of Proposition \ref{proposition-adaptive-1}}.
The term at the left-hand side of inequality (\ref{est1}) captures,
in addition to the aforementioned estimation of the total number of infected nodes
in the network (counting repetition) over time, the convergence rate of the adaptive control strategy.
This allows us to draw the following insights:
(i) The larger $\rho$, the faster the
virus spreading will die out;
(ii) the larger degree of initial infection,
the slower the virus spreading will die out;
(iii) the larger edge infection probability $\gamma$,
the slower the virus spreading will die out.

\subsubsection{Adaptive control under which the virus spreading will not die out but will be contained to a desired level}
\label{sec:adaptive-control-2}

In the above we have given some sufficient condition on adjusting the defense
or $\beta_v(t)$ so that the virus spreading will die out.
What if the required $\beta_v(t)$ cannot be achieved, meaning that
we may not expect that the virus spreading die out?
This is possible because
the defense may not be as good as one may wish or because of budget limitation.
In this case, we ask an alternative interesting question:
What it takes so that $i_v(t)$ can converge or be contained
to some pre-determined level of infection $\cc_{v}^{*}$?

\begin{theorem}
\label{theorem-control-to-specified-state}
\emph{(characterization of adaptive control strategy under which the virus spreading will be contained to a desired level of infection)}
Consider the following variant of the master equation Eq. (\ref{Eq.2}),
\begin{eqnarray}
\label{eq-adaptive-control-2}
\frac{d\cc_{v}(t)}{dt}&=&\bigg[1-\prod_{(u,v)\in E}[1-\gamma
\cdot \cc_{u} {(t)}]\bigg](1-\cc_{v} {(t)})-\beta_{v}(t)\cc_{v} {(t)}+w_{v}
\end{eqnarray}
For any $1\ge \cc_{v}^{*}> 0$, $v=\oneton$, letting
\begin{eqnarray*}
\beta_{v}^{*}=\frac{[1-\prod_{(u,v)\in E}(1-\gamma \cdot
\cc_{u}^{*})](1-\cc_{v}^{*})}{\cc_{v}^{*}},
\end{eqnarray*}
if we use the following adaptive control strategy
\begin{eqnarray*}
\frac{d\beta_{v}(t)}{dt}=\rho[\cc_{v}(t)-\cc_{v}^{*}]\cc_{v}(t), ~~~~~~
w_{v}=\eta[\cc_{v}^{*}-\cc_{v}(t)],
\end{eqnarray*}
where $\rho$ is an (almost) arbitrary positive constant, and $\eta$
is a positive constant with
$\eta+\min\limits_{v}\beta_{v}^{*}>1+(1-\gamma)\lambda_{1}$ where
$\lambda_{1}$ is the largest (in modulus) eigenvalue of the
adjacency matrix $A$, we have
$\lim_{t\to\infty}\cc_{v}(t)=\cc_{v}^{*}$.
\end{theorem}

\begin{proof}
From Perron-Frobenius theorem \cite{Berman03}, we have that there
exist some positive constants $P_{1},\cdots,P_{n}$ such that $P={\rm
diag}[P_{1},\cdots,P_{n}]$ satisfies
$[PA+A^{\top}P]\le\lambda_{1}P$. Consider the candidate Lyapunov
function with respect to $\cc=[\cc_{1},\cdots,\cc_{n}]^{\top}$ and
$\tilde{\beta}=[\beta_{1},\cdots,\beta_{n}]^{\top}$:
\begin{eqnarray*}
V(\cc,\tilde{\beta})=\frac{1}{2}\sum_{v=1}^{n}P_{v}[\cc_{v}-\cc_{v}^{*}]^{2}+\frac{1}{2\rho}\sum_{v=1}^{n}P_{v}
(\beta_{v}-\beta_{v}^{*})^{2}.
\end{eqnarray*}
Differentiating $V(x,\tilde{\beta})$ gives
\begin{eqnarray*}
&&\frac{d{V}(\cc,\tilde{\beta})}{dt}\\
&=&\sum_{v=1}^{n}P_{v}[\cc_{v}(t)-\cc_{v}^{*}]\bigg\{[1-\prod_{(u,v)\in
E}(1-\gamma \cdot \cc_{u} {(t)})](1-\cc_{v}
{(t)})-\beta_{v}(t)\cc_{v} {(t)}-\eta(\cc_{v}-\cc_{v}^{*})\bigg\}\\
&&+\sum_{v=1}^{n}P_{v}(\beta_{v}(t)-\beta_{v}^{*})(\cc_{v}(t)-\cc_{v}^{*})\cc_{v}(t)\\
&=&\sum_{v=1}^{n}P_{v}[\cc_{v}(t)-\cc_{v}^{*}]\bigg\{\big[1-\prod_{(u,v)\in
E}(1-\gamma \cdot \cc_{u} {(t)})\big](1-\cc_{v} {(t)}) \\
&&-\big[1-\prod_{(u,v)\in E}(1-\gamma \cdot \cc_{u}^{*}
)\big](1-\cc_{v}^{*})
-\beta_{v}(t)\cc_{v} {(t)}-\beta_{v}^{*}\cc_{v}^{*}+\eta(\cc_{v}-\cc_{v}^{*})\bigg\}\\
&&+\sum_{v=1}^{n}P_{v}(\beta_{v}(t)+\beta_{v}^{*})(\cc_{v}(t)-\cc_{v}^{*})\cc_{v}(t).
\end{eqnarray*}
Note that
\begin{eqnarray*}
&&\bigg|[1-\prod_{(u,v)\in E}(1-\gamma \cdot \cc_{u}
{(t)})](1-\cc_{v} {(t)})-[1-\prod_{(u,v)\in E}(1-\gamma
\cdot \cc_{u}^{*} )](1-\cc_{v}^{*})\bigg|\\
&\le&\bigg|(\cc_{v}(t)-\cc_{v}^{*})\prod_{(u,v)\in
E}(1-\gamma\cc_{v}(t))\bigg| +\bigg|(1-\cc_{v}^{*})\big[\prod_{(u,v)\in
E}(1-\gamma\cc_{u}(t))-\prod_{(u,v)\in
E}(1-\gamma\cc_{u}^{*})\big]\bigg|\\
&\le&|\cc_{v}(t)-\cc_{v}^{*}|+\sum_{(u,v)\in
E}|\cc_{u}(t)-\cc_{u}^{*}|\prod_{(u_{1},v)\in
E,~u_{1}>u}|1-\gamma\cc_{u_{1}}^{*}|\prod_{(u_{2},v)\in
E,~u_{2}<u}|1-\gamma\cc_{u_{2}}(t)|\\
&\le&|\cc_{v}(t)-\cc_{v}^{*}|+\gamma\sum_{(u,v)\in
E}|\cc_{u}(t)-\cc_{u}^{*}|.
\end{eqnarray*}
Thus, we have
\begin{eqnarray*}
\frac{d{V}(\cc,\tilde{\beta})}{dt} &\le&\sum_{v=1}^{n}
P_{v}|\cc_{v}(t)-\cc_{v}^{*}|^{2}+\sum_{v=1}^{n}\sum_{u=1}^{n}\gamma
P_{v}a_{vu}|\cc_{v}(t)-\cc_{v}^{*}||\cc_{u}(t)-\cc_{u}^{*}|\\
&&-\sum_{v=1}^{n}\beta_{v}^{*}P_{v}(\cc_{v}(t)-\cc_{v}^{*})^{2}
-\eta\sum_{v=1}^{n}P_{v}[\cc_{v}(t)-\cc_{v}^{*}]^{2}.
\end{eqnarray*}
Let $z(t)=[z_{1}(t),\cdots,z_{n}(t)]^{\top}$ with
$z_{v}(t)=|\cc_{v}(t)-\cc_{v}^{*}|$. Since $\beta_{v}^{*}\ge 0$, we have
\begin{eqnarray*}
\frac{d{V}(\cc,\tilde{\beta})}{dt}&\le& z^{\top}\big[
P+\gamma(PA+A^{\top}P)/2-\eta P-B^{*}P\big]z\\
&\le&
z^{\top}\big[1+\gamma\lambda_{1}-\eta-\min_{v}\beta_{v}^{*}\big]P
z<0~,\forall~z\ne 0,
\end{eqnarray*}
where $B^{*}={\rm diag}[\beta_{1},\cdots,\beta_{n}]$. Due to the
LaSalle principle \cite{LaSalle1960}, we have that the system will
converge to the largest invariant set in
$\{(\cc,\tilde{\beta}):~\sum_{v=1}^{n}[\cc_{v}-\cc_{v}^{*}]^{2}=0\}$,
which implies that $\lim_{t\to\infty}\cc_{v}(t)=\cc_{v}^{*}$ for all
$v=\oneton$. \qed
\end{proof}

\noindent{\bf Discussion}. The above theorem is quite general because of the term $w_v$
in Eq. (\ref{eq-adaptive-control-2}). If $w_v \neq 0$, the adaptive control strategy must be used with caution
because we must guarantee that its value does have physical meanings. In general, $w_v \le 0$ would be reasonable;
in our simulation study (Section \ref{sec:simulation-study-alpha-eq-0}),
we set $w_v = 0$ for simplicity. On the other hand, the theorem is necessarily
based on the premise that $\gamma(t)=\gamma$, namely that $\gamma(t)$ does not vary
with respect to time, because of the way $\beta_v^*$ is defined.
In our simulation study (Section \ref{sec:simulation-study-alpha-eq-0}), we will show that the result is quite robust, meaning
that even if $\gamma(t)$ varies with respect to time (as we considered in Section \ref{sec:implicit-adaptive-case}),
the result is still valid.
This is very important because the fixed $\gamma$ can be seen as, in a sense, the average of the unknown
$\gamma(t)$ over time.

Similar to Proposition \ref{proposition-adaptive-1}, we can have

\begin{proposition}
\label{prop:adaptive-control-2}
We have
\begin{eqnarray}
\int_{0}^{\infty}\sum_{v=1}^{n}|\cc_{v}(\tau)-\cc_{v}^{*}|d\tau
&\le&\frac{1}{\min_{v}P_{v}(\eta-1-\gamma\lambda_{1})}
\big[\frac{1}{2}\sum_{v=1}^{n}P_{i}|\cc_{v}(0)-\cc_{v}^{*}|^{2} \nonumber\\
&&+\frac{1}{2\rho}\sum_{v=1}^{n}
|\beta_{v}(0)-\beta_{v}^{*}|^{2}\big]\label{est3}.
\end{eqnarray}
\end{proposition}

\noindent{\bf Physical meanings of Proposition \ref{prop:adaptive-control-2}}. The above proposition offers the following insights:
The larger $\rho$, the larger $\eta$, the smaller $\lambda_{1}$,
the smaller $\gamma$, the smaller differential between $\cc_{v}(0)$ and
$\cc_{v}^{*}$, or the smaller the differential between $\beta_{v}(0)$ and $\beta_{v}^{*}$,
the faster the virus spreading will die out.

\section{Simulation Study}
\label{sec:simulation-study-alpha-eq-0}

We conduct simulation to complement our analytic study for two purposes.
First, we want to confirm our analytical results offered in Section \ref{sec:model-alpha-eq-0}.
Second, we want to draw some relevant observations that are not offered by our analytic results.
Such observations may guide future studies of analytic models (e.g., how to enhance them so
that other useful insights may be obtained analytically).

As mentioned above, our model is very general because it accommodates dynamical graph topology $G(t)=(V,E(t))$.
However, it's not clear at this stage
how to appropriately define a physically meaningful way according to which
the topology changes. Therefore we leave the full-fledged characterization (beyond what is implied
by our analytical results) of the impact of dynamical topology to future work.
We conducted simulations using both synthetic (regular, random, and power-law) graphs
and a real network graph. Due to space limitation, here we report the simulation results in the latter
case (but all the simulation results are consistent). The real network graph $G=(V,E)$
is based on the Oregon router views (available from {\tt http://topology.eecs.umich.edu/data.html}),
where $|V|=11,461$ representing AS peers, $|E|=32,730$ representing links between the AS peers.
The largest eigenvalue of the corresponding adjacency matrix is $\lambda_1=75.2407$.

\subsection{Methodology}

Our simulation is conducted in an event-driven fashion.
For the purpose of studying the dynamics under our adaptive control strategies, we need to measure $i_v(t)$,
the probability that node $v\in V$ is infected at time $t$. This parameter can be
obtained in a tedious way (i.e., by conducting for example 100 simulation runs in parallel, rather than in sequence,
because we need to count the number of times each node $v$ is infected at each time step).
A much more simpler way however is to use our model
formula Eq. (\ref{Eq.2}) to compute $i_v(t)$ instead, as long as the model is accurate.
To confirm the accuracy of the model, we compare it with simulation.
For simplicity, we let both $\beta(t)$ and $\gamma(t)$ be some periodic functions with period $T$, which means that
both attack and defense vary with respect to time. We consider
three settings: $\beta(t)$ and $\gamma(t)$ being synchronous, asynchronous, or anti-synchronous
because we want to observe whether, and if so to what extent, the degree of (a)synchrony has an impact on the outcome.
Specifically, we consider two sets of parameters:
$\beta(t)\in \{0.3,0.5\}$ and $\gamma(t) \in \{0.003,0.007\}$
according to the functions shown in Figure \ref{fig:beta-gamma-t-model};
$\beta(t)\in \{0.005,0.015\}$ and $\gamma(t) \in \{0.003,0.007\}$ in the same fashion.
In the asynchronous case, we let $\beta_v(t)$ is $T/4$ behind $\gamma_v(t)$
because cure often comes after attack is identified.
To draw insights into whether the period $T$ has an impact on the outcome, we
consider $T = 8, ~16$, respectively.
In any case, it is clear that $\beta(t)$ is an implicit function of $\gamma(t)$.

\begin{figure*}[ht]
\centering
\subfigure[synchronous setting]
{\label{fig:beta-gamma-t-model-synchronous-case}
\includegraphics[width=.32\textwidth]{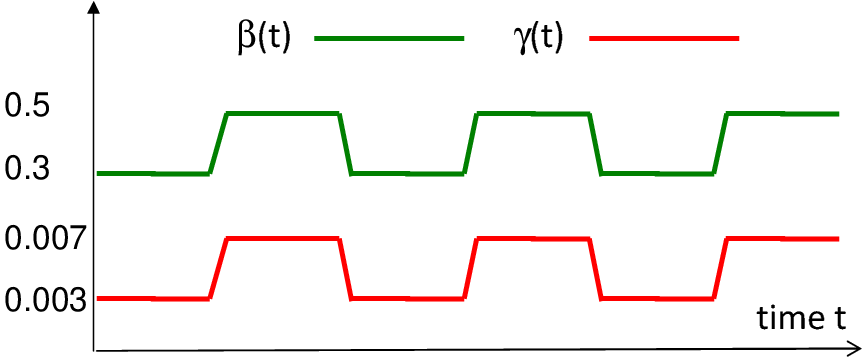}}
\subfigure[asynchronous setting]
{\label{fig:beta-gamma-t-model-asynchronous-case}
\includegraphics[width=.32\textwidth]{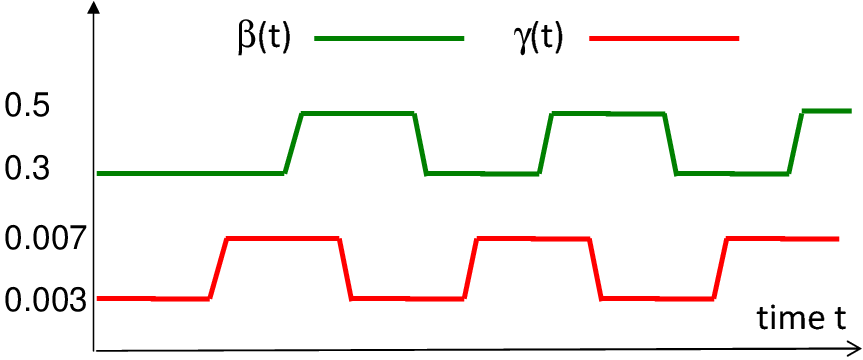}}
\subfigure[anti-synchronous setting]
{\label{fig:beta-gamma-t-model-antisynchronous-case}
\includegraphics[width=.32\textwidth]{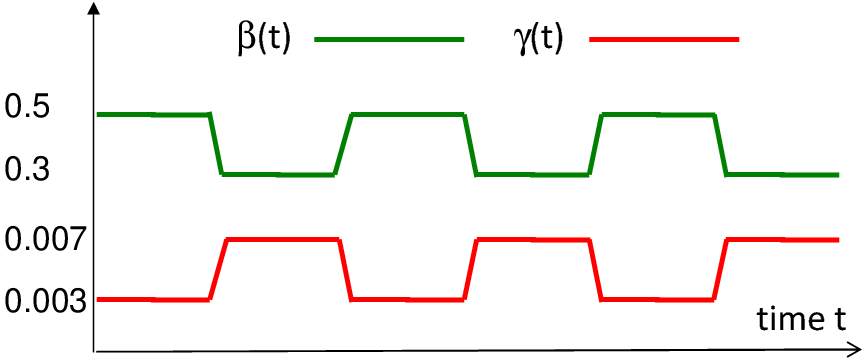}}
\caption{Examples of synchronous, asynchronous, and counter-synchronous $\beta(t)$ and $\gamma(t)$}
\label{fig:beta-gamma-t-model}
\end{figure*}

Figure \ref{fig:Model-accuracy} plots the curves obtained by simulation and by model computing
in the case of both $\beta_v(t)$ and $\gamma_v(t)$ have period $T=8$ and $T=16$ (as shown in Figure \ref{fig:beta-gamma-t-model}).
In each graph, we let the virus initially infect 2,292 or 20\% vertices that are randomly selected;
note that the degree of initial infection does not impact whether the virus spreading will die out or not.
Since the model computing and simulation results (obtained as the average of 50 simulation runs)
match almost perfectly no matter the virus spreading will die out or not,
we will use simulation and model computing interchangeably. Since the same phenomenon applies to
both cases of $T=8$ and $T=16$, in what follows we only report the case of $T=8$.
Again, the accuracy result allows us to obtain $i_v(t)$ via model computing in the process of confirming
the analytical results of our adaptive control strategies.

\begin{figure*}[ht]
\centering
\subfigure[T=8]
{\label{fig:model-accuracy-T=8}
\includegraphics[width=.48\textwidth]{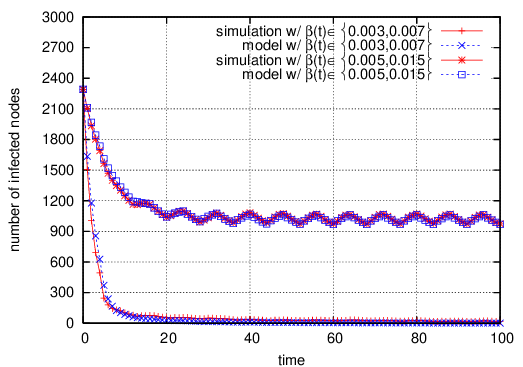}}
\subfigure[T=16]
{\label{fig:model-accuracy-T=16}
\includegraphics[width=.48\textwidth]{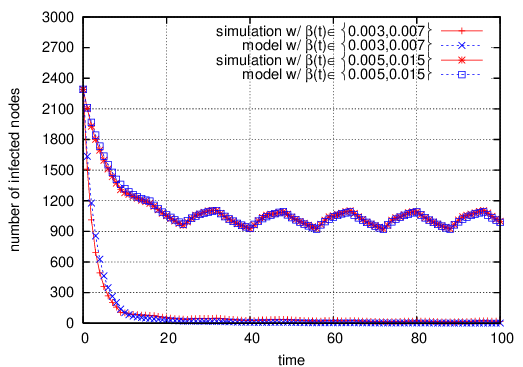}}
\caption{Model accuracy}
\label{fig:Model-accuracy}
\end{figure*}

\subsection{Confirmation of the sufficient conditions in the semi-adaptive scenario}

In this section, we use the aforementioned Oregon graph to confirm
our analytical results presented in Section \ref{sec:implicit-adaptive-case}.
We confirm Theorem \ref{thm4} and Corollary \ref{theorem-main-result-2} because they
offer succinct sufficient conditions under which the virus spreading dies out.

\subsubsection{Confirmation of Theorem \ref{thm4}}

For the graph, we let the virus initially infect 2,292 or 20\% randomly selected nodes,
and consider three cases between the model's input parameter --- synchrony, asynchrony, and anti-synchrony
as illustrated in Figure \ref{fig:beta-gamma-t-model}.
Since Theorem \ref{thm4} has two parts, we confirm them respectively.

\begin{figure*}[ht]
\centering
\subfigure[Parameter set (i)]
{\label{fig:global-analysis-regular-graph-case-1-virus-th4-part1-rl4-1}
\includegraphics[width=.48\textwidth]{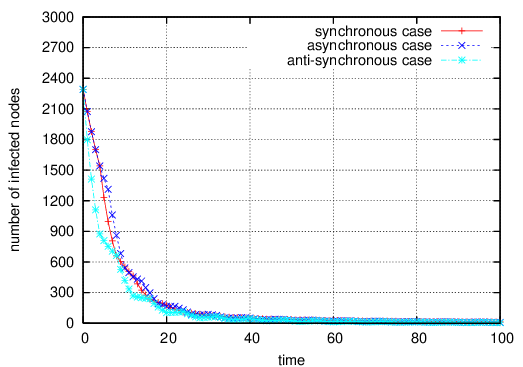}}
\subfigure[Parameter set (ii)]
{\label{fig:global-analysis-random-graph-case-1-virus-th4-part1-rl4-2}
\includegraphics[width=.48\textwidth]{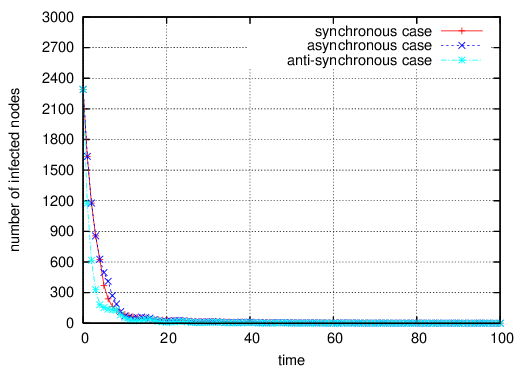}}
\caption{Confirmation of sufficient condition under which virus spreading dies out ($T = 8$)}
\label{fig:J-1-Theorem-4-case1-T=8}
\end{figure*}

\noindent{\bf Case 1: Confirmation of the sufficient condition under which the virus spreading will die out}.
We consider two sets of parameters:
(i) $\beta(t)\in \{0.3,0.5\}$ and $\gamma(t) \in \{0.003,0.007\}$;
(ii) $\beta(t)\in \{0.1,0.22\}$ and $\gamma(t) \in \{0.001,0.003\}$.
Both functions, $\beta(t)$ and $\gamma(t)$, have period $T=8$.
Both parameter sets satisfy the sufficient condition of Theorem \ref{thm4}, namely $\lambda_1 < \bar{\beta}/\bar{\gamma}$,
which means that the virus spreading will die out.
Figure \ref{fig:J-1-Theorem-4-case1-T=8} plots the dynamics of the numbers of infected nodes with respect to time.
From Figure \ref{fig:J-1-Theorem-4-case1-T=8} we can draw the following observations.
First, the virus spreading does die out at about the 50th and 25th step, respectively, which confirms the sufficient condition
under which the virus spreading will die out. It is an interesting future work to quantitatively
characterize how the speed of convergence (i.e., dying out) depends upon functions $\beta(t)$ and $\gamma(t)$.
Second, it is counter-intuitive and interesting that the virus spreading is somewhat more effectively
defended against in the anti-synchronous case than in the synchronous case, which is in turn more
effectively defended against than in the asynchronous case. More studies
are needed in order to explain this phenomenon.
Third, the curves are convex, meaning that cure is more effective in the early stage
of the attack-defense dynamics than in the later stage. For example, it takes
a shorter period of time to reduce the infection from 2,292 nodes to 100 nodes
than to reduce the infection from 100 nodes to zero nodes (i.e., dying out).

\smallskip

\noindent{\bf Case 2: Confirmation of the condition under which the virus spreading may not die out}.
We consider two sets of parameters:
(i) $\beta(t)\in \{0.2,0.4\}$ and $\gamma(t) \in \{0.003,0.007\}$;
(ii) $\beta(t)\in \{0.05,0.15\}$ and $\gamma(t) \in \{0.001,0.003\}$.
Both functions, $\beta(t)$ and $\gamma(t)$, have period $T=8$.
Both parameter sets do not satisfy the sufficient condition of Theorem \ref{thm4}
because $\lambda_1 > \bar{\beta}/\bar{\gamma}$, which means that the virus spreading does not die out
in some initial infection configurations.

\begin{figure*}[ht]
\centering
\subfigure[Parameter set (i)]
{\label{fig:global-analysis-regular-graph-case-1-virus-th4-part2-rl4-1}
\includegraphics[width=.48\textwidth]{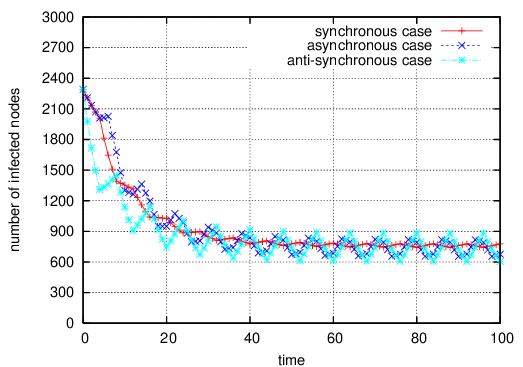}}
\subfigure[Parameter set (ii)]
{\label{fig:global-analysis-random-graph-case-1-virus-th4-part2-rl4-2}
\includegraphics[width=.48\textwidth]{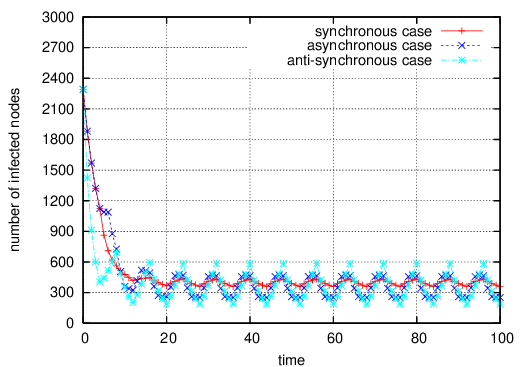}}
\caption{Confirmation of virus spreading not dying out ($T = 8$)}
\label{fig:J-1-Theorem-4-case2-T=8}
\end{figure*}

Figure \ref{fig:J-1-Theorem-4-case2-T=8} plots the dynamics, from which we draw the following observations.
First, the virus spreading does not die out, which confirms Theorem \ref{thm4}.
Second, all the curves exhibit periodic behaviors, but the extent of oscillation in the
case of anti-synchrony is more significant than in the case of asynchrony, which
in turn is more significant than in the case of synchrony. This means that the degree of
synchrony between $\beta(t)$ and $\gamma(t)$ will impact the outcome when the virus spreading does not
die out. Third, comparing Figures \ref{fig:global-analysis-regular-graph-case-1-virus-th4-part2-rl4-1}
and \ref{fig:global-analysis-random-graph-case-1-virus-th4-part2-rl4-2}, we observe that,
under the same synchrony, the outcome will depend on functions $\beta(t)$ and $\gamma(t)$.
More studies are needed to characterize these dependence relationships.

\subsubsection{Confirmation of Corollary \ref{theorem-main-result-2}}

Corollary \ref{theorem-main-result-2} gives an even more succinct sufficient condition
under which the virus spreading will die out. For the graph, we let the virus initially infect
2,291 or 20\% randomly selected nodes.
For the case the sufficient condition in Corollary \ref{theorem-main-result-2} is satisfied,
we consider two sets of parameters that are uniformly chosen at random from certain intervals:
(i)  $\beta(t)\in [0.7, 0.9]$ and $\gamma(t) \in [0.006, 0.0014]$;
(ii) $\beta(t) \in [0.1, 0.3]$ and $\gamma(t) \in [0.0015, 0.0035]$.
In each parameter setting, the sufficient condition stated in Corollary \ref{theorem-main-result-2} is satisfied,
namely $\lambda_1 < \E(\beta(0))/\E(\gamma(0))$, which means that the virus spreading will die out.

For the case the sufficient condition in Corollary \ref{theorem-main-result-2} is not satisfied,
we consider two sets of parameters that are uniformly chosen at random from certain intervals:
(i) $\beta(t)  \in [0.4, 0.6]$ and   $\gamma(t) \in [ 0.006, 0.014]$;
(ii) $\beta(t) \in [0.05, 0.15]$ and $\gamma(t) \in [0.0015, 0.0035]$.
In each parameter setting, the sufficient condition stated in Corollary \ref{theorem-main-result-2} is
not satisfied because $\lambda_1 > \E(\beta(0))/\E(\gamma(0))$,
and thus the analytical result says that the virus spreading does not die out
in some initial infection configurations.

\begin{figure*}[ht]
\centering
\subfigure[The sufficient condition is satisfied]
{\label{fig:global-analysis-regular-graph-case-1-virus-th5-part1-J1-rl4}
\includegraphics[width=.48\textwidth]{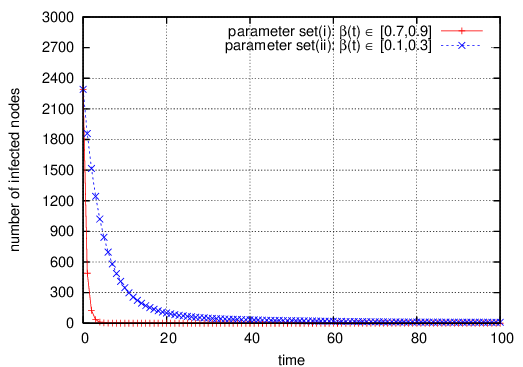}}
\subfigure[The sufficient condition is not satisfied]
{\label{fig:global-analysis-random-graph-case-1-virus-th5-part2-J1-rl4}
\includegraphics[width=.48\textwidth]{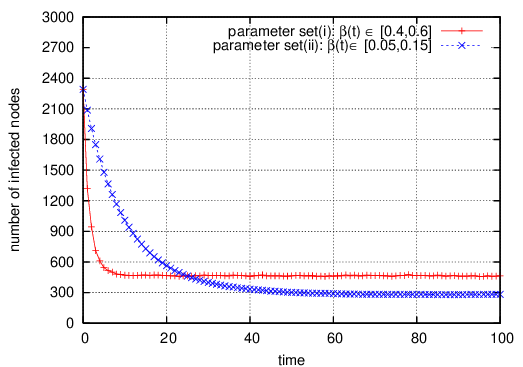}}
\caption{Confirmation of Corollary \ref{theorem-main-result-2} ($T = 8$)}
\label{fig:J-1-corollary-1}
\end{figure*}

Figure \ref{fig:global-analysis-regular-graph-case-1-virus-th5-part1-J1-rl4} plots the
dynamics of the number of infected nodes with respect to time when the sufficient condition is satisfied.
From it we can draw the following observations.
First, the virus spreading does die out, as predicted by Corollary \ref{theorem-main-result-2}.
Second, the larger the $\beta(t)$, the more effective the defense against the virus spreading.
Third, all the curves are convex, meaning that it takes a shorter period of time
to significantly reduce the number of infected nodes (e.g., from 2,291 to 50)
than to making the virus spreading die out (e.g., from 50 to zero).

Figure \ref{fig:global-analysis-random-graph-case-1-virus-th5-part2-J1-rl4}
plots the dynamics of the number of infected nodes with respect to time when the sufficient condition is not satisfied.
From it we can draw the following observations.
First, the virus spreading does not die out, which confirms Corollary \ref{theorem-main-result-2}.
Second, the larger (in a stochastic sense) the $\beta(t)$, the earlier the system will converge to the steady state.
However, the ultimate degree of infection does not depend on $\beta(t)$, but rather on
$\E(\beta(0))/\E(\gamma(0))$, which means that $\E(\beta(0))/\E(\gamma(0))$ may be used as an indicator of
steady-state infection when the virus spreading does not die out. It is an interesting future work
to rigorously characterize this phenomenon.

\subsection{Confirmation of the controllability in the fully-adaptive scenario}
\label{sec:confirmation-adaptive-control}

\subsubsection{Confirmation of Theorem \ref{thm_stablize1}}

Theorem \ref{thm_stablize1} states that even if we do not know $\gamma(t)$ but we
may be able to observe $i_v(t)$ and may be able to adjust the defense as needed,
following its control strategy will cause the dying out of the virus spreading.
To compare the effects of adaptive control and semi-adaptive control,
in our simulation study, we also used the periodical functions $\beta(t) \in \{0.375, 0.40\}$
and $\gamma(t) \in \{0.003, 0.007\}$
with period $T=8$ as illustrated in Figure \ref{fig:beta-gamma-t-model}.
These parameters satisfy the sufficient condition in Theorem \ref{thm4}, which means
that the virus spreading will die out as we discussed above.
To ensure comparability, we also let the virus initially infect 2,291 or 20\% randomly selected nodes.
As mentioned before, since the anti-synchronous defense is somewhat
more effective than the synchronous and asynchronous defenses,
we will compare it with the outcome of the adaptive control strategy.

\begin{figure*}[ht]
\centering
\subfigure[Fully-adaptive control]
{\label{fig:fully-adaptive control-th4}
\includegraphics[width=.48\textwidth]{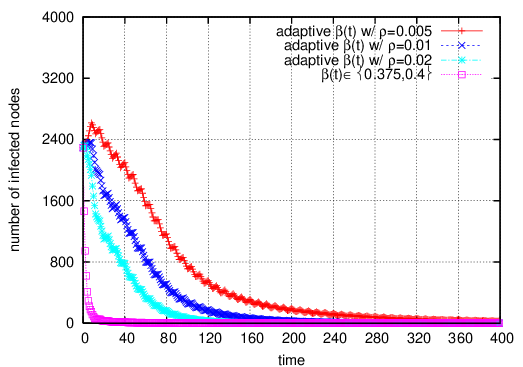}}
\subfigure[Comparison of cure capabilities]
{\label{fig:cure capabilities-th4}
\includegraphics[width=.48\textwidth]{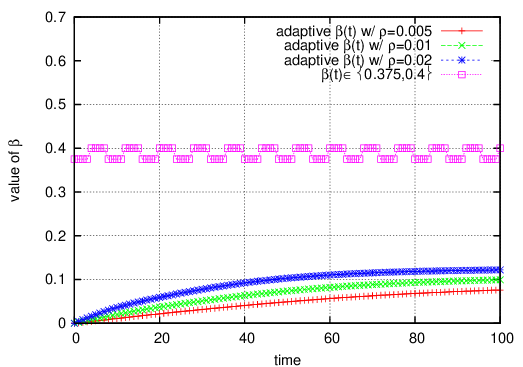}}
\caption{Confirmation of semi-adaptive dying out vs. fully-adaptive control}
\label{fig:confirming-adaptive-control-1}
\end{figure*}

Figure \ref{fig:fully-adaptive control-th4} plots the dynamics of the number of infected nodes
with respect to time in the following four cases: the adaptive control parameter $\rho=0.005$;
the adaptive control parameter $\rho=0.01$; the adaptive control parameter $\rho=0.02$;
the comparison dynamics corresponding to the anti-synchronous case
with $T=8$ periodical function $\beta(t) \in \{0.375, 0.40\}$.
We draw the following observations.
First, $\rho$ plays a crucial role in indicating the rate at which the virus spreading dies out.
For example, for $\rho=0.02$, it takes only about 80 steps to reduce the number of infected nodes from 2,291 to 160
(nevertheless it takes another 60 steps to kill the virus spreading, namely to reduce the number of infected nodes from 160 to zero);
for $\rho=0.01$, it takes about 130 steps to reduce the number of infected nodes from 2,291 to 160
(nevertheless it takes about another 100 steps to kill the virus spreading, namely to reduce the number of infected
nodes from 160 to zero).
This also confirms the physical meanings of Proposition \ref{proposition-adaptive-1} discussed above.

Second, Figure \ref{fig:fully-adaptive control-th4} indicates that
for all $\rho=0.005$, $\rho=0.01$ and $\rho=0.02$, the fully-adaptive defenses are less effective than
the semi-adaptive defense represented by $\beta(t) \in \{0.375, 0.40\}$.
As we show in Figure \ref{fig:cure capabilities-th4},
this is caused by the fact that the semi-adaptive $\beta(t)$ is much larger than the adaptive $\beta(t)$.
This means that the sufficient condition in the semi-adaptive case, under which the virus spreading dies out,
may be significantly beyond being necessary. In contrast, the fully-adaptive control strategy is much more
``cost-effective" because larger $\beta(t)$ will likely cause a higher cost.

\subsubsection{Confirmation of Theorem \ref{theorem-control-to-specified-state}}

Theorem \ref{theorem-control-to-specified-state} states that even if we do not know $\gamma(t)$ but we
may be able to observe $i_v(t)$ and may be able to adjust the defense (but cannot kill the virus spreading),
then following its control strategy will cause the containment of the virus spreading.
In our simulation study, we used the periodical function $\gamma(t) \in \{0.0005,0.001\}$
with period $T=8$ similar to what was shown in Figure \ref{fig:beta-gamma-t-model}. This input parameter is not used
in our adaptive control algorithm, rather it is merely for the purpose of comparison
to the sufficient condition in Theorem \ref{thm4}, which requires the $\beta_v(t)$
satisfy certain property (for example, we use $\beta(t) \in \{0.01, 0.02\}$, meaning that the virus spreading does not
die out as predicated),
and the $\beta_v(t)$ derived from our adaptive control strategy.
To ensure the comparability, we also let the virus initially infect 2,291 or 20\% randomly selected nodes.
As mentioned before, since the anti-synchronous defense is more effective than the synchronous and asynchronous defenses,
we will compare it with the outcome of the adaptive control algorithm.

\begin{figure*}[ht]
\centering
\subfigure[Fully-adaptive control]
{\label{fig:fully-adaptive control-th5}
\includegraphics[width=.48\textwidth]{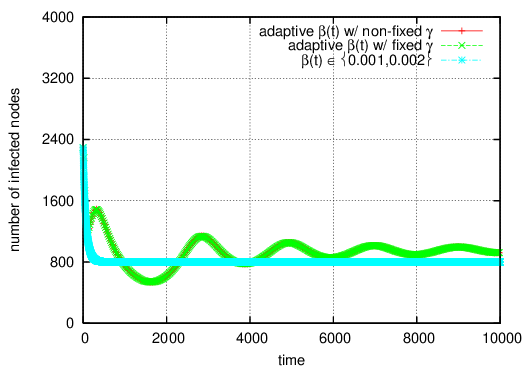}}
\subfigure[Comparison of cure capabilities]
{\label{fig:cure capabilities-th5}
\includegraphics[width=.48\textwidth]{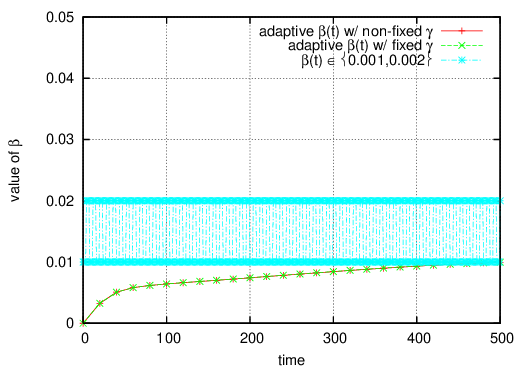}}
\caption{Confirmation of the virus spreading containment via fully-adaptive control (the red-color curves are
hidden behind the green-color curves)}
\label{fig:confirming-adaptive-control-2}
\end{figure*}

In our simulation we set $i^*=0.1$ (i.e., we want to contain the degree of infection to 10\%)
and $u_v=0$ (i.e., a special case of Theorem \ref{theorem-control-to-specified-state}).
Figure \ref{fig:fully-adaptive control-th5} plots the dynamics of the number of infected nodes
with respect to time in the following three cases: the adaptive control parameter $\rho=0.001$;
the adaptive control parameter $\rho=0.001$ with fixed $\bar{\gamma}(t) = (0.0005+0.001)/2=0.00075$
as specified in Theorem \ref{theorem-control-to-specified-state});
the comparison dynamics corresponding to the anti-synchronous case
with $T=8$ periodical function $\beta(t) \in \{0.01, 0.02\}$.
We draw the following observations.
First, the control strategy does contain the infection to the pre-determined level of
$i^*=0.1$ or 10\% infection. Moreover,
the adaptive control strategy is robust because perturbation in $\gamma$ does not
fundamentally change the dynamics behavior.
This also confirms the physical meanings of Proposition \ref{prop:adaptive-control-2} discussed above.

Second, Figure \ref{fig:fully-adaptive control-th5} indicates
that the adaptive defenses are slightly less effective in defending against the virus spreading
than the defense of $T=8$ periodical function $\beta(t) \in \{0.01, 0.02\}$.
As we show in Figure \ref{fig:cure capabilities-th5}, the adaptive
control strategy can be much more ``cost-effective" because it leads to significantly smaller $\beta(t)$.

\section{Conclusion}
\label{sec:conclusion}

We have presented a novel dynamical systems model for studying both semi-adaptive and fully-adaptive
defenses against virus spreading. For semi-adaptive defense, we give general as well as succinct
sufficient conditions under which the virus spreading will die out. For fully-adaptive defense, we characterize two
adaptive control strategies under which the virus spreading will die out or will be contained
to a desired level of infection.
Our analytical results are confirmed with simulation study.

This paper brings a range of open questions for future research. In addition to those
mentioned in the body of the paper, here are more examples:
What are the necessary conditions under which the virus spreading will die out?
What are the optimal adaptive control strategies?

\smallskip

\noindent{\bf Acknowledgement}. We thank the anonymous reviewers for their useful comments,
and Raj Boppana for helpful discussion on the simulation.

This work was supported in part by AFOSR, AFOSR MURI, ONR, and UTSA.
The views and conclusions contained in the article are those of the authors and should
not be interpreted as, in any sense, the official policies or endorsements of the government or
the agencies.


\begin{thebibliography}{}

\bibitem[\protect\citeauthoryear{Anderson and May}{Anderson and
  May}{1991}]{Anderson1991}
{\sc Anderson, R.} {\sc and} {\sc May, R.} 1991.
\newblock {\em Infectious Diseases of Humans}.
\newblock Oxford University Press.

\bibitem[\protect\citeauthoryear{Arnold}{Arnold}{1998}]{Arn}
{\sc Arnold, L.} 1998.
\newblock {\em Random Dynamical Systems}.
\newblock Springer-Verlag.

\bibitem[\protect\citeauthoryear{Bailey}{Bailey}{1975}]{Bailey1975}
{\sc Bailey, N.} 1975.
\newblock {\em The Mathematical Theory of Infectious Diseases and Its
  Applications}.
\newblock 2nd Edition. Griffin, London.

\bibitem[\protect\citeauthoryear{Berman and Shaked-Monderer}{Berman and
  Shaked-Monderer}{2003}]{Berman03}
{\sc Berman, A.} {\sc and} {\sc Shaked-Monderer, N.} 2003.
\newblock {\em Completely positive matrices}.
\newblock World Scientific Publishing.

\bibitem[\protect\citeauthoryear{Bhargava, Dilley, and Riedl}{Bhargava
  et~al\mbox{.}}{1986}]{BhargavaSIGOPS86}
{\sc Bhargava, B.}, {\sc Dilley, J.}, {\sc and} {\sc Riedl, J.} 1986.
\newblock Raid: a robust and adaptable distributed system.
\newblock In {\em ACM SIGOPS European Workshop}.

\bibitem[\protect\citeauthoryear{Chakrabarti, Wang, Wang, Leskovec, and
  Faloutsos}{Chakrabarti et~al\mbox{.}}{2008}]{WangTISSEC08}
{\sc Chakrabarti, D.}, {\sc Wang, Y.}, {\sc Wang, C.}, {\sc Leskovec, J.}, {\sc
  and} {\sc Faloutsos, C.} 2008.
\newblock Epidemic thresholds in real networks.
\newblock {\em ACM Trans. Inf. Syst. Secur.\/}~{\em 10,\/}~4, 1--26.

\bibitem[\protect\citeauthoryear{Ganesh, Massoulie, and Towsley}{Ganesh
  et~al\mbox{.}}{2005}]{TowsleyInfocom05}
{\sc Ganesh, A.}, {\sc Massoulie, L.}, {\sc and} {\sc Towsley, D.} 2005.
\newblock The effect of network topology on the spread of epidemics.
\newblock In {\em Proceedings of IEEE Infocom 2005}.

\bibitem[\protect\citeauthoryear{Hethcote}{Hethcote}{2000}]{HethcoteSIAMRew00}
{\sc Hethcote, H.} 2000.
\newblock The mathematics of infectious diseases.
\newblock {\em SIAM Rev.\/}~{\em 42,\/}~4, 599--653.

\bibitem[\protect\citeauthoryear{Kephart and White}{Kephart and
  White}{1991}]{KephartOkland91}
{\sc Kephart, J.} {\sc and} {\sc White, S.} 1991.
\newblock Directed-graph epidemiological models of computer viruses.
\newblock In {\em IEEE Symposium on Security and Privacy}. 343--361.

\bibitem[\protect\citeauthoryear{Kephart and White}{Kephart and
  White}{1993}]{KephartOkland93}
{\sc Kephart, J.} {\sc and} {\sc White, S.} 1993.
\newblock Measuring and modeling computer virus prevalence.
\newblock In {\em IEEE Symposium on Security and Privacy}. 2--15.

\bibitem[\protect\citeauthoryear{Kermack and McKendrick}{Kermack and
  McKendrick}{1927}]{Kermack1927}
{\sc Kermack, W.} {\sc and} {\sc McKendrick, A.} 1927.
\newblock A contribution to the mathematical theory of epidemics.
\newblock {\em Proc. of Roy. Soc. Lond. A\/}~{\em 115}, 700--721.

\bibitem[\protect\citeauthoryear{LaSalle}{LaSalle}{1960}]{LaSalle1960}
{\sc LaSalle, J.} 1960.
\newblock Some extensions of liapunov's second method.
\newblock {\em IRE Trans. Circuit Theory\/}~{\em 7}, 520--527.

\bibitem[\protect\citeauthoryear{McKendrick}{McKendrick}{1926}]{McKendrick1926}
{\sc McKendrick, A.} 1926.
\newblock Applications of mathematics to medical problems.
\newblock {\em Proc. of Edin. Math. Soceity\/}~{\em 14}, 98--130.

\bibitem[\protect\citeauthoryear{Oseledec}{Oseledec}{1968}]{Ose}
{\sc Oseledec, V.} 1968.
\newblock A multiplicative ergodic theorem. characteristic ljapunov, exponents
  of dynamical systems.
\newblock {\em English Translation. Trans. Moscow Math. Soc.\/}~{\em 19},
  197--231.

\bibitem[\protect\citeauthoryear{Pesin}{Pesin}{1977}]{Pesin77}
{\sc Pesin, Y.~B.} 1977.
\newblock Characteristic lyapunov exponents and smooth ergodic theory.
\newblock {\em Russ. Math. Surv.\/}~{\em 32}, 55--114.

\bibitem[\protect\citeauthoryear{Wang, Chakrabarti, Wang, and Faloutsos}{Wang
  et~al\mbox{.}}{2003}]{WangSRDS03}
{\sc Wang, Y.}, {\sc Chakrabarti, D.}, {\sc Wang, C.}, {\sc and} {\sc
  Faloutsos, C.} 2003.
\newblock Epidemic spreading in real networks: An eigenvalue viewpoint.
\newblock In {\em Proc. of the 22nd IEEE Symposium on Reliable Distributed
  Systems (SRDS'03)}. 25--34.

\bibitem[\protect\citeauthoryear{Zou, Duffield, Towsley, and Gong}{Zou
  et~al\mbox{.}}{2005}]{ZouSRUTI05}
{\sc Zou, C.}, {\sc Duffield, N.}, {\sc Towsley, D.}, {\sc and} {\sc Gong, W.}
  2005.
\newblock Adaptive defense against various network attacks.
\newblock In {\em Workshop on Steps to Reducing Unwanted Traffic on the
  Internet (SRUTI'05)}. 69--75.

\end{thebibliography}


\end{document}